\documentclass[10pt]{amsart}
\usepackage[latin1]{inputenc}
\usepackage{amsfonts}
\usepackage{amsmath}
\usepackage{amssymb}
\usepackage{amsthm}
\usepackage{fontenc}
\usepackage{graphicx}
\usepackage[margin=1.5in]{geometry}

\usepackage[all]{xy}
\newtheorem{theorem}{Theorem}[section]

\newtheorem{proposition}[theorem]{Proposition}

\newtheorem{corollary}[theorem]{Corollary}
\newtheorem{definition}{Definition}
\newtheorem{remark}{Remark}
\newtheorem{conjecture}{Conjecture}
\newtheorem{question}{Question}

\numberwithin{equation}{section}

\DeclareMathOperator{\sgn}{sgn}
\DeclareMathOperator{\Det}{Det}
\DeclareMathOperator{\Trace}{Trace}

\title{Measure Equipartitions via Finite Fourier Analysis}

\author{Steven Simon}

\begin{document}

\maketitle

\begin{abstract} 
      
Applications of harmonic analysis on finite groups are introduced to measure partition problems, with equipartitions obtained as the vanishing of prescribed Fourier transforms. For elementary abelian groups $Z_p^k$, $p$ an odd prime, equipartitions are by $k$-tuples of complex regular $p$-fans in $\mathbb{C}^d$, analogues of the famous Gr\"unbaum problem on equipartitions in $\mathbb{R}^d$ by $k$-tuples of hyperplanes (i.e., regular 2-fans). Here the number of regions is a prime power, as usual in topological applications to combinatorial geometry. For general abelian groups, however, the Fourier perspective yields new classes of equipartitions by families of complex regular fans $F_{q_1},\ldots, F_{q_k}$ (such as those of a ``Makeev-type"), including when the number of regions is not a prime power. 

\end{abstract}

\section{Introduction} 

	A common problem in combinatorial and computational geometry concerns equipartitions of measures on Euclidian space. Given any collection of absolutely continuous measures $\mu_1,\ldots, \mu_m$ on $\mathbb{R}^d$ (simply to be called \textit{measures} from now on), one seeks a partition $\{\mathcal{R}_1,\ldots, \mathcal{R}_n\}$ of $\mathbb{R}^d$ by a fixed class of ``nice" geometric regions, each of which contains an equal fraction of each total measure. The most famous such problem, dating back to Gr\"unbaum [10], asks for the smallest dimension $d=\Delta(m,k)$ for which any $m$ measures on $\mathbb{R}^d$ can be equipartitioned into $2^k$ orthants determined by $k$ hyperplanes (see, e.g., [3, 5, 11, 17, 21, 31]). In particular, the widely applied Ham Sandwich Theorem -- any $d$ measures on $\mathbb{R}^d$ can be simultaneously bisected by a single hyperplane -- is $\Delta(d,1)=d$.
	
	Equipartitions are ordinarily obtained topologically. Owing to some, often implicit group symmetry on each partition, the problem can be reduced to an equivariant framework to which the vast machinery of classical algebraic topology -- equivariant cohomology and index theory, characteristic and other obstruction classes, spectral sequences, cobordism theory, and so on -- can be applied. See, e.g., [18, 32--34], for a general survey of these methods. 
		
\subsection{A Harmonic Analysis Approach} The central objective of this paper is to propose more systematic application of  Fourier analysis on finite groups to these measure equipartition problems and their topological reductions (see [23] for its use in the Topological Tverberg Problem, another central problem in combinatorial geometry, and [25] for an application of compact groups to transversality-type theorems).  For given $G$, we consider a class of naturally indexed partitions $\{\mathcal{R}_g\}_{g\in G}$ by convex domains. If each $\mu_i$ is \textit{complex-valued} (i.e., a pair of real measures) any decomposition determines maps $\mathcal{F}_{\mu_i}: G\rightarrow \mathbb{C}$, $g\mapsto \mu_i(\mathcal{R}_g)$, with  Fourier expansions (see, e.g., [26]) \begin{equation} \mu_i(\mathcal{R}_g) = \sum_{\sigma \in \hat{G}}n_\sigma\, \Trace(c_{i,\sigma} \sigma_g), \end{equation} where $\hat{G}$
consists of all non-isomorphic irreducible  unitary representations of $G$,  $n_\sigma$ is the dimension of the representation, and the  $c_{i,\sigma}$ are the matrix-valued Fourier transforms \begin{equation} c_{i,\sigma}=\frac{1}{|G|}\sum_{g\in G} \mu_i(\mathcal{R}_g)\sigma_g^{-1} \in M(\mathbb{C},n_\sigma)\end{equation}	

	 Our measure partitions are obtained as the vanishing of prescribed transforms of the above expansions.  As discussed in Section 2, these partitions are examples of a purely group-theoretic family of generalizations of the Ham Sandwich Theorem, first introduced in [24], whereby measures are ``balanced" by a group's linear representations.  Many partition problems previously considered can be put in this ``$G$-Ham Sandwich" context, notably the Gr\"unbaum problem above.
	 
	 Owing to this representation theory setting, it is not surprising that our results are obtained via the usual computational methods of topological combinatorics -- here, the calculation of the total Chern class in group cohomology of the considered representation (Section 7) -- nor is it surprising that the equivariant topological techniques themselves (e.g., ideal-valued index theory as used in [3, 17, 32, 34]) can be recast in the language of Fourier analysis. Nonetheless, the emphasis on Fourier transforms introduces a  novel perspective which opens new possibilities for applications, as we now discuss. 

\subsection{Summary of Results} We shall be primarily concerned (Section 3--5) with finite abelian groups $G=Z_{q_1}\times\ldots \times Z_{q_k}$ and partitions of $\mathbb{C}^d$ by complex regular fans $F_{q_1},\ldots, F_{q_k}$. For the elementary abelian groups $G=Z_p^k$, $p$ an odd prime, one has equipartitions by $k$-tuples of complex regular $p$-fans with complex affine independent centers (Theorem 4.1, Corollaries 4.1--4.2), analogues of the Gr\"unbaum problem above (i.e., equipartitions by regular 2-fans). In these cases, the number of regions is a prime power, as is typical for equivariant topological applications to measure partition theory (see, e.g., [2, 3, 4, 14, 16, 17, 28, 31], et cetera). 

	The main advantage of the Fourier method is the great latitude in choice of transforms to be annihilated, so that, for general $G$, judicious selections yield a wider variety of equipartitions than those previously considered. This is exemplified by the ``Makeev-type" results of Section 5, which include when the number of regions is no longer a prime power: if $q_j=pr_j$ for an odd prime $p$, one has a complex regular fan partition $F_{q_1},\ldots, F_{q_k}$ such that each of the $r=r_1\cdots r_k$ sub $k$-tuples of regular $p$-fans equipartitions each measure  (Theorem 5.1, Corollaries 5.1--5.2). Moreover, Proposition 5.1 gives a case where the equipartitions occur by pairs of fans with distinct $q_1$ and $q_2$. Finally, a non-abelian example is provided in Section 6 for the quaternion group $Q_8$. 
	
\section{$G$-Ham Sandwich Theorems} 

As above, let $G$ be a finite group, which we suppose gives rise to class of partitions $\mathbb{R}^d=\cup_g \mathcal{R}_g$ by convex domains which are invariant under a free $G$-action. These partitions can be called ``regular," since they often arise as conical partitions associated to regular convex polytopes. Given any representation $\rho: G\rightarrow O(n)$ and any $n$-tuple $\mu=(\mu_1,\ldots, \mu_n)$ of measures on $\mathbb{R}^d$, one can then consider the ``$(\rho, G)$-average"
 		 \begin{equation} \frac{1}{|G|} \sum_{g\in G} \rho_g^{-1}(\mu(\mathcal{R}_g)) =\frac{1}{|G|} \sum_{g\in G} g^{-1}\cdot \mu(\mathcal{R}_e\cdot g) \in \mathbb{R}^n \end{equation}  of the measures of the regions of any such $G$-decomposition. In a general sense, the sum evaluates  the symmetry of the measures of the $\{\mathcal{R}_g\}_{g\in G}$ with respect to the given representation, so we say that 
		 
\begin{definition} A $n$-tuple of measures $\mu=(\mu_1,\ldots, \mu_n)$ is $``(\rho,G)$-balanced" by the partition $\{\mathcal{R}_g\}_{g\in G}$ if the average (2.1) is zero.\end{definition} 

By a ``$(\rho,G)$-Ham Sandwich Theorem," we mean a result which guarantees that any $n$-tuple of measures on $\mathbb{R}^d$ can be simultaneously balanced by the representation $\rho$. The annihilation of Fourier coefficients in (1.1) is a unitary case of this construction. In particular, \textit{unitary abelian Ham Sandwich Theorems are equivalent to the vanishing of prescribed Fourier transforms}, since any unitary representation of an abelian group $G$ is the direct sum of 1-dimensional ones (see, e.g., [22]) and the balancing of $\mu_i$ by $\sigma: G\rightarrow U(1)$ means $c_{i,\sigma}=0$. For non-abelian groups, the vanishing of $c_{1,\sigma_1}, \ldots, c_{k,\sigma_k}$ is the $(\rho,G)$-balancing of the $n:=\sum_i n^2_{\sigma_i}$-tuple $\mu=(\mu_1\textbf{e}_1,\ldots, \mu_1\textbf{e}_{n_{\sigma_1}},\ldots, \mu_k\textbf{e}_{n_{\sigma_k}})$ by $\rho=\oplus_{i=1}^k n_{\sigma_i}\sigma_i$, where $\textbf{e}_i\in \mathbb{C}^n$ is the $i$-th standard basis vector (and the trivial measure is denoted by 0).

Before giving new equipartitions, we first show how the Gr\"unbaum Problem is a $Z_2^k=\{\pm1\}^k$-case of this $G$-Ham Sandwich/Fourier partition scheme. 
	
\subsection{The Gr\"unbaum Problem}  As in [3, 5], we consider the standard $Z_2^k$-action on $(S^d)^k$. Each $x_j=(\mathbf{a}_j,b_j)\in S^d$, $\|\mathbf{a}_j\|^2+|b_j|^2=1$,  gives a unique hyperplane $H_j=  \{\mathbf{u}\in\mathbb{R}^d\mid \langle \mathbf{u}, \mathbf{a}_j \rangle = b_j\}$ if $\mathbf{a}_j\neq 0$ and a hyperplane ``at infinity" otherwise. The $Z_2^k$-orbits $\{\delta\cdot x\}_{\delta\in Z_2^k}$ of all $x=(x_1,\ldots, x_k)\in (S^d)^k$ therefore produce all the partitions of $\mathbb{R}^d$ by the (not necessarily distinct, some possibly empty) regions  $\mathcal{O}_\delta = \{\mathbf{u}\in\mathbb{R}^d\mid \, (\forall\, 1\leq j \leq k)\,(\exists \, v_j\geq0)\, \langle \mathbf{u},\mathbf{a}_j\rangle - b_j = \delta_jv_j\}$ determined by $k$ or fewer (genuine) hyperplanes in $\mathbb{R}^d$, $\delta=(\delta_1,\ldots, \delta_k) \in Z_2^k$.
	
	On the other hand, $\hat{Z_2^k}=\mathbb{Z}_2^k$ and each $\chi_\epsilon: Z_2^k\rightarrow U(1)$ is real, so each $\mu_i(\mathcal{O}_\delta)=\sum_{\epsilon\in \mathbb{Z}_2^k} c_{i,\epsilon} \chi_\epsilon$ is real-valued if the measures are. As each transform $c_{i,0}$ of the trivial representation is $2^{-k}\sum_\delta \mu_i(\mathcal{O}_\delta) = 2^{-k}\mu_i(\mathbb{R}^d)$, annihilating all other $c_{i,\epsilon}$ -- i.e., $(\rho,Z_2^k)$-balancing the $m(2^k-1)$-tuple $\mu=(\mu_1,\ldots,\mu_1,\ldots, \mu_m,\ldots, \mu_m)$ by $\rho=m\oplus_{\epsilon\neq0}\chi_\epsilon$ -- is  the equipartition of each $\mu_i$ by what must therefore be $2^k$ distinct orthants of $k$ hyperplanes in general position. It should be observed that the use of $\rho$ is equivalent to that of the regular representation $\mathbb{R}^m[Z_2^k]$ in the usual topological reduction of this problem (see, e.g., [3, 5, 17]), but ignores the full Weyl-group $Z_2^k\rtimes S_k$ action on $(S^d)^k$ also considered there, $S_k$ being the symmetric group. Thus the newer $k=2$ results $\Delta(2^{n+2}+2, 2)=3\cdot 2^{n+1} +2$ of [31], which arise from an effective use of the dihedral group $D_8$ (also used in [5]), are not recovered by our construction.

\section{Fourier Partitions by Complex Regular Fans}

 	Suppose now that $G= Z_{q_1}\times\ldots\times Z_{q_k}$ is an arbitrary finite abelian group, where each cyclic group $Z_q$ is identified with the $q$-th roots of unity  $\{\zeta_q^k\}_{k=0}^{q-1}$, $\zeta_q=\exp(2\pi i/q)$. A natural class of partitions here are by $k$ or fewer complex regular fans $F_{q_1},\ldots, F_{q_k}$ in $\mathbb{C}^d$, where each $F_q$ is the union of $q$ half-hyperplanes whose successive dihedral angles are all equal to $2\pi/q$ and have as their boundaries a common complex hyperplane.  
	
	Explicitly, let $R_r(q)=\{v\in \mathbb{C}\mid \arg (v) \in [(r-1)/q, (r+1)/q]\}$ denote the regular $q$-sectors of $\mathbb{C}$ centered at the origin, $0\leq r<q$, and let $\langle \mathbf{u},\mathbf{v} \rangle_{\mathbb{C}}= \sum_{i=1}^d u_i\bar{v}_i$ be the standard Hermitian form on $\mathbb{C}^d$. If $\mathbf{a}_j\neq 0$, the $Z_{q_j}$-orbit $\{\zeta_{q_j}^r x_j\}_{j=0}^{r-1}$ of $x_j=(\mathbf{a}_j,b_j)\in S(\mathbb{C}^{d+1})$, $\|\mathbf{a}_j\|^2 +|b_j|^2=1$, partitions $\mathbb{C}^d$ into the sectors $\mathcal{S}_{r_j}(q_j) = \{\mathbf{u}\in\mathbb{C}^d\mid (\exists v_j\in R_{r_j}(q_j))\, \langle \mathbf{u},\mathbf{a}_j \rangle_{\mathbb{C}} - \bar{b}_j = v_j\}$ of the complex regular $q_j$-fan 
\begin{equation} F_{q_j} =  \{\mathbf{u}\in\mathbb{C}^d\mid (\exists \, 0\leq r <q_j )\, \langle \mathbf{u},\mathbf{a}_j \rangle_{\mathbb{C}} - \bar{b}_j = \zeta_{q_j}^r \} \end{equation} 
\noindent centered about the complex hyperplane $H^{\mathbb{C}}_j=\{\mathbf{u}\in\mathbb{C}^d\mid \langle \mathbf{u},\mathbf{a}_j \rangle_{\mathbb{C}} = \bar{b}_j\}$, and we say as before that $F_{q_j}$  is a centered ``at infinity" if $\mathbf{a}_j=0$. Thus the $G$-orbits of $[S(\mathbb{C}^{d+1})]^k$ under the standard action yield all partitions $\{\mathcal{R}_g:=\cap_{j=1}^k \mathcal{S}_{r_j}(q_j)\}_{g\in G}$ of $\mathbb{C}^d$ by all $F_{q_1},\ldots, F_{q_k}$. We make the following definition:  

\begin{definition} A collection $F_{q_1},\ldots, F_{q_k}$ in $\mathbb{C}^d$ is called a non-trivial complex fan partition if at least one of the fans is not centered at infinity.\end{definition} 
  
 	On the other hand, one has the identification $\hat{G}=\oplus_{j=1}^k\mathbb{Z}_{q_j}$, given explicitly by $\chi_{\mathbf{\epsilon}}(g)= \Pi_{j=1}^k \zeta_{q_j}^{\epsilon_j}$, $\epsilon=(\epsilon_1,\ldots, \epsilon_k)\in \oplus_{j=1}^k\mathbb{Z}_{q_j}$, so the Fourier expansion (1.1) takes the simple form \begin{equation} \mu_i(\mathcal{R}_g)=\sum_{\epsilon \in  \oplus_{j=1}^k \mathbb{Z}_{q_j}} c_{i,\epsilon} \chi_\epsilon(g) \end{equation}

\noindent One then has the following $G$-Ham Sandwich theorem: 

\begin{theorem} Let $\rho=\oplus_{r=1}^n \chi_{\epsilon_r}:G\rightarrow U(n)$, $\epsilon_r=(\epsilon_{r,1},\ldots, \epsilon_{r,k}) \in \oplus_{j=1}^k \mathbb{Z}_{q_j}$. If \begin{equation} f(b_1,\ldots, b_k)=\Pi_{r=1}^n(\epsilon_{r,1}b_1+\ldots + \epsilon_{r,k}b_k)\in \mathbb{Z}[b_1,\ldots, b_k]/(q_1b_1,\ldots, q_kb_k) \end{equation}  is not contained in the ideal \, $\mathcal{I}= (b_1^{d+1}, \ldots, b_k^{d+1})$, then for any complex measures $\mu_1,\ldots,\mu_n$ on $\mathbb{C}^d$, there exists a non-trivial complex fan partition $F_{q_1},\ldots, F_{q_k}$ with $c_{i,\epsilon_i}=0$ in (3.2) for each $1\leq i \leq n$.  \end{theorem}	

	We defer the proof of Theorem 3.1 to Section 7, preferring instead to first give applications to real measures.  For now, we note that the theorem reduces to Proposition 7.1 on equivariant maps, itself proved by a calculation of the top Chern class of the given representation.

\section{Complex Gr\"unbaum Problems} 

As a real hyperplane is a regular 2-fan, one has a natural complex generalization of the classical Gr\"unbaum problem: 
	
\begin{question} What is the minimum $d=\Delta_{\mathbb{C}}(m;q_1,\ldots, q_k)$, denoted $\Delta_{\mathbb{C}}(q;m,k)$ if  $q_j=q$ for all $j$, for which any $m$ measures on $\mathbb{C}^d$ can be equipartitioned by $Q=\Pi_{j=1}^kq_j$ regions determined by $k$ complex regular $q_j$-fans? \end{question}	

	In the original Gr\"unbaum problem, the lower bound $\Delta(m,k)\geq m(2^k -1)/k$ (shown in [21], conjectured there to be optimal, and established as such in [17, 21, 31] in a number of cases) follows by considering $m$ disjoint segments on the moment curve $M=\{(t,t^2,\ldots, t^d)\mid t\in\mathbb{R}\}$: $k$ equipartitioning  hyperplanes give $m(2^k-1)$ points of intersection with $M$, hence $m(2^k-1)$ roots to $k$ polynomials of degree $d$, so $kd\geq m(2^k-1)$. A similar approach using points on the complex moment curve $M_{\mathbb{C}}=\{(z,z^2,\ldots, z^d)\mid z\in \mathbb{C}\}$ gives a lower bound here, at least when $k=1$:

\begin{proposition} For $q>2$, \begin{equation} \Delta_{\mathbb{C}}(q;1,m) \geq m\lfloor (q-1)/2\rfloor  \end{equation}  \end{proposition}

\begin{proof} If $m$ measures on $\mathbb{C}^d$ are equipartitioned by a complex regular $q$-fan, then the interior of the union of any two adjacent sectors contains at most $2/q$ of each total measure. Consider point collections $C_1,\ldots, C_m$, each consisting of $q'=\lfloor (q-1)/2\rfloor$ points of $M_{\mathbb{C}}$, and let $\mu_i^\varepsilon$ be the volume of the union of the $\varepsilon$-balls with centers the points of $C_i$. A standard limiting argument as in [27, 30] shows that if $d=\Delta_{\mathbb{C}}(q;1,m)$, then there must exist a complex regular $q$-fan for which the interior of the union of any two adjacent sectors again contains at most $2/q$ points of each $C_i$, hence none. Thus each point lies on the fan, and in fact in its center since the interior of a half-hyperplane is contained in the interior of the union of two adjacent sectors. Hence $d\geq q'm$, since a point of intersection of $M_{\mathbb{C}}$ and a complex hyperplane represents a root of a degree $d$ polynomial. \end{proof} 

\noindent We conjecture a similar lower bound for  $\Delta_{\mathbb{C}}(m;q_1,\ldots, q_k)$ for all $k\geq 1$:  

\begin{conjecture} Let $Q=\Pi_{j=1}^k q_j$, $q_j>2$. Then \begin{equation*} k\Delta_{\mathbb{C}}(m;q_1,\ldots, q_k) \geq m\lfloor (Q-1)/2 \rfloor \end{equation*} \end{conjecture} 

\subsection{Upper Bounds via Theorem 3.1} 
	
	By assuming the measures are real-valued, it follows that their transforms satisfy $c_{i,{\mathbf{-\epsilon}}}=\overline{c_{i,{\mathbf{\epsilon}}}}$. Moreover, $c_{i,0}=\mu_i(\mathbb{C}^d)/Q$ as before, so letting \begin{center} $ \lceil \Pi  \rceil: =\{\epsilon \in \oplus_{j=1}^k\mathbb{Z}_{q_j}-\{0\}$ with last non-zero coordinate $\epsilon_j \leq \lceil(q_j-1)/2\rceil \}$, \end{center}	
	it is easily seen that the equipartition of each $\mu_1,\ldots, \mu_m$ is the vanishing of each $c_{i,\epsilon}$ with $\epsilon\in \lceil \Pi \rceil$. Note that the $F_{q_1},\ldots, F_{q_k}$ must all be genuine and that their centers must be complex affine independent in this circumstance. The associated polynomial (3.3) is 
\begin{equation} f(b_1,\ldots, b_k)=\Pi_{\epsilon\in \lceil \Pi \rceil} (\epsilon_1b_1+\ldots +\epsilon_kb_k)^m \end{equation} 

\noindent As a simple application, one has 

\begin{proposition} \begin{equation} \Delta_{\mathbb{C}}(9;1,1)=4 \end{equation} 
\end{proposition}

\begin{proof} $f(b_1)=4!\cdot b_1^4\not \in (b_1^5) \subset \mathbb{Z}_9[b_1]$, so $\Delta_{\mathbb{C}}(9;1,1)\leq 4$. Thus $\Delta_{\mathbb{C}}(9;1,1)=4$ by (4.1).
\end{proof}

\begin{remark} It is perhaps interesting to observe that, since $6 \equiv 0 \pmod{3}$, Proposition 4.2 is obtained as a cohomology class representing a zero divisor of $\mathbb{Z}_q$ which is zero in $\mathbb{Z}_p$  for any prime $p$ dividing $q$.  Along with Propositions 5.1 and 6.1 below and the recent values $\Delta(2^{n+2} +1, 2) = 3\cdot 2^{n+1}+2$ of [31] (representing $2\in\mathbb{Z}_4$, but not obtained either via ideal-valued cohomological index theory or characteristic classes), these are the first equipartitions to be obtained in such fashion.\end{remark}
	
	For the elementary abelian groups $Z_p^k$, $p$ an odd prime, the polynomial (4.2) has coefficients in the field $\mathbb{Z}_p$. Resulting upper bounds on $\Delta_{\mathbb{C}}(p;m,k)$ then strongly parallel Theorem 4.1 of [17], obtained from $Z_2^k$-cohomological index theory with $\mathbb{Z}_2$-coefficients, which is still the best known general result on $\Delta(m,k)$. As there, one has a naturally related Dickson polynomial \begin{equation} D(p,k) =  \Det\begin{pmatrix} b_1 &  \ldots & b_k \\
                            b_1^p & \ldots & b_k^p \\
                            \vdots &  & \vdots \\
                            b_1^{p^{k-1}} &\ldots & b_k^{p^{k-1}}  
\end{pmatrix} = \sum_{\sigma \in S_k} \sgn(\sigma) b_{\sigma(1)}b_{\sigma(2)}^p \cdots b_{\sigma(k)}^{p^{k-1}}\end{equation}

\begin{theorem}  $\Delta_{\mathbb{C}}(p;m,k) \leq d$ if $D(p,k)^{m(p-1)/2} \not\in (b_1^{d+1},\ldots, b_k^{d+1})$ if $p$ is an odd prime.

\end{theorem}

\begin{proof} $D(p,k)$ is the product of all non-zero $\epsilon_1b_1+\ldots +\epsilon_kb_k$ whose last non-zero coordinate is 1 (see the proof of Proposition 1.1 of [29]), so  $f(b_1,\ldots, b_k)=[(\frac{p-1}{2})!]^m D(p,k)^{m(p-1)/2}$. \end{proof}

\noindent In particular, $\Delta_{\mathbb{C}}(p;m,1)\leq m(p-1)/2$ (see also [24]), so by (4.1) 

\begin{corollary} 
\begin{equation} \Delta_{\mathbb{C}}(p;m,1)=m(p-1)/2\end{equation} for all odd primes $p$. 
\end{corollary} 

	For $k>1$, our best results  occur when $k=2$ (as is true for $\Delta(m,k)$). Note that one has the exact value $\Delta_{\mathbb{C}}(p;m,2)=m(p^2-1)/4$ in the following, provided the conjectured lower bound holds.  	
	
\begin{corollary} Let $\sum_{i=1}^na_ip^i$ be the base $p$ expansion of $m(p-1)/2$, $p$ an odd prime. If  each $a_i$ is even, then
\begin{equation} \Delta_{\mathbb{C}}(p;m,2)\leq m(p^2-1)/4 \end{equation}
\end{corollary}
\begin{proof} 

Let $m'=m(p-1)/2$. The sum of the exponents of any monomial in $D(p,2)^{m'}=(b_1b_2^p-b_2b_1^p)^{m'}$ is $m(p^2-1)/2$, so we seek one of degree $m(p^2-1)/4$. The unique such monomial is  $b_1^{m(p^2-1)/4}b_2^{m(p^2-1)/4}$, and by Lucas's theorem (see, e.g., [9]) this coefficient is $\binom{m'}{m'/2}=\Pi_{i=1}^n\binom{a_i}{a_i/2}\neq 0$.
\end{proof} 

	 For example, any two measures on $\mathbb{C}^4$ can be equipartitioned by a pair of complex regular 3-fans (actually, $3\leq \Delta_{\mathbb{C}}(3;2,2)\leq 4$, where the lower bound comes from dimension considerations), and any measure on $\mathbb{C}^6$ can be equipartitioned by a pair of complex regular 5-fans. Note that $\Delta_{\mathbb{C}}(p; 2(p^{n+1}-1)/(p-1),2)\leq \frac{p+1}{2}\cdot (p^{n+1}-1)$ follows by setting each $a_i=p-1$. In particular, $\Delta_{\mathbb{C}}(3; 3^{n+1}-1,2)\leq 2\cdot (3^{n+1} -1)$, which can again be compared to the optimal value $\Delta(2^{n+1}-1,2)\leq \frac{3}{2}\cdot (2^{n+1} -1)$ of [17].   
 
\section{Modulo Equipartitions of a Makeev Type}

	When $G\neq Z_p^k$, $p$  prime, the polynomial arising from $\rho=m\oplus_{\epsilon\neq 0} \chi_\epsilon$ vanishes whenever $m\geq 1$ unless $m=1$ and $G=Z_4$, while for $m$ odd the polynomial (4.2) from $\rho=m\oplus_{\epsilon\in \lceil \Pi\rceil} \chi_\epsilon$ vanishes whenever $m\geq 1$ unless $m=1$ and $G=Z_4$ or $Z_9$, so there is no hope of equipartition via Theorem 3.1 in these cases. 

	Nonetheless, a variety of other equipartitions follow by annihilating different prescribed transforms. We highlight one such class of results. Supposing that $q_j=p_jr_j$,  annihilating each $c_{i,\epsilon}$ in (3.2) except when $\epsilon \in \oplus_{j=1}^kp_j\mathbb{Z}_{q_j}$ yields an equipartition of each $\mu_i$ ``modulo" the subgroup $H=Z_{p_1}\times \ldots \times Z_{p_k}$: \begin{equation} \mu_i(\mathcal{R}_{hg})=\mu_i(\mathcal{R}_g) \end{equation} for each $g\in G$ and each $h\in H$, since the remaining $\chi_\epsilon$ are trivial on $H$. Thus there  is a collection $F_{q_1},\ldots,F_{q_k}$ of complex regular $q_j$-fans (necessarily distinct and in general position), each of whose $r=\Pi_{j=1}^kr_j$ sub-collections of regular $p_j$-fans $F_{p_1}\subset F_{q_1},\ldots, F_{p_k}\subset F_{q_k}$ equipartitions each measure. Such partitions are similar in spirt to those of Makeev [16] and their generalizations [3], in which there exist $n$ orthogonal hyperplanes, any $k$ of which equipartition a given set of measures. 
				
\begin{theorem} [Complex Makeev] Let $q_1=pr_1, \ldots, q_k=pr_k$, $p$ an odd prime, let $r=\Pi_{j=1}^k r_j$, and let $D(p,k)\in\mathbb{Z}_p[b_1,\ldots, b_k]$ be the Dickson polynomial (4.4). If $D(p,k)^{rm(p-1)/2}\notin (b_1^{d+1},\ldots, b_k^{d+1})$, then for any $m$ measures on $\mathbb{C}^d$ there exists $k$ complex regular $p$-fans $F_{q_1}, F_{q_2},\ldots, F_{q_k}$, each of whose $r$ sub-collections of $k$ regular $p$-fans equipartitions each measure.
\end{theorem} 

\begin{proof}  For $G=\Pi_{j=1}^kZ_{q_j}$, let $\lceil \Pi \rceil _r = \{\epsilon \notin \oplus_{j=1}^kp\mathbb{Z}_{q_j}$ with last non-zero coordinate $\epsilon_j\leq \lceil(q_j-1)/2\rceil\}$. The associated polynomial (3.3) is $f=g^m$, where $g=\Pi_{\epsilon\in \lceil \Pi \rceil_r}(\epsilon_1b_1+\ldots +\epsilon_kb_k)$, so $f$ is a non-zero constant multiple of $D(p,k)^{rm(p-1)/2}$ when reduced mod $p$.\end{proof} 

\noindent One has the following corollaries for odd primes $p$ as in the non-modulo cases: 

\begin{corollary} For $q=pr$, any $m$ measures on $\mathbb{C}^{m(q-r)/2}$ can be equipartitioned by each of the $r$ regular $p$-fans of some complex regular $q$-fan. \end{corollary}

\begin{corollary} Let $mr(p-1)/2=\sum_{i=1}^n a_i p^i$, where each $a_i$ is even. Then any $m$ measures on $\mathbb{C}^{mr(p^2-1)/4}$ can be equipartitioned by each of the $r=r_1r_2$ regular $p$-fans contained in some pair of complex regular fans $F_{pr_1}$ and $F_{pr_2}$.  
\end{corollary}

For example, although $\Delta_{\mathbb{C}}(15;m,1)\geq 7m$, Corollary 5.1 shows that any $m$ measures on $\mathbb{C}^{5m}$ (respectively, $\mathbb{C}^{6m}$), can be equipartitioned by each of the 5 regular 3-fans (respectively, 3 regular 5-fans) of some complex regular 15-fan. By Corollary 5.2, any two measures on $\mathbb{C}^{16}$ can be equipartitioned by each of the four pairs of regular 3-fans of a pair of complex regular 6-fans.\\

 We give one final Makeev-type result, noteworthy in that $q_1\neq q_2$ for the equipartitioning fans $F_{q_1}$ and $F_{q_2}$. On the other hand, owing to multiplication mod 9, the dimension is large compared to the conjectured lower bound $\Delta_{\mathbb{C}}(1;9,2)\geq 20$: 

\begin{proposition} For any measure on $\mathbb{C}^{27}$, there exists a pair of complex regular 9-fans $F_9^1$ and $F_9^2$ such that $F_9^1$ and each regular 3-fan of  $F_9^2$ equipartitions the measure.
\end{proposition}  

\begin{proof} We annihilate each $c_\epsilon$ with $\epsilon \in \lceil \Pi \rceil$ and $\epsilon\neq (0,3), (3,3), (6,3)$. The expansion is $\mu\left(\mathcal{R}_{(\zeta_9^{k_1},\zeta_9^{k_2})}\right) = c_0 + 2Re(c_{(0,3)}\zeta_3^{k_2}) + 2Re(c_{(3,3)}\zeta_3^{k_1+k_2})+2Re(c_{(3,6)}\zeta_3^{k_1+2k_2})$, and the corresponding polynomial \begin{center} $f(b_1,b_2)=6b_1^{10}b_2^9(b_1^2-b_2^2)^3(b_1^2-4b_2^2)^3(b_1^2-16b_2^2)^3=6b_1^{10}b_2^9(b_1^{18}-b_2^{18})\in \mathbb{Z}_9[b_1,b_2]$ \end{center} does not lie in $(b_1^{28},b_2^{28})$. As the regions determined by $F_9^1$ and and the regular 3-fans of $F_9^2$ are exactly $\mathcal{S}=\cup_{j=0}^2\mathcal{R}_{(\zeta_9^{k_1},\zeta_9^{k_2+j})}, 0\leq k_1,k_2< 9$, summing the expansion over $0\leq j\leq 2$ gives $\mu(\mathcal{S})=\mu(\mathbb{C}^{27})/27$. 
\end{proof} 

 \section{Equipartitions by Pairs of Cubical Wedges}

	For a non-abelian example, we consider the quaternion group $Q_8=\{\pm1,\pm i, \pm j, \pm k\}$. One has a corresponding canonical partition of the Quaternions $\mathbb{H}\cong\mathbb{R}^4$ by cones $V_g=\cup_{r\geq 0}rC_g$ on the faces (i.e., cubes) $C_g=\{w\in P \mid \langle w, g\rangle_{\mathbb{R}}=1\}$ which form the boundary of the 8-cell (4-cube) $P=\{w\in\mathbb{H}\mid \langle w, g\rangle_{\mathbb{R}}\leq 1\, \forall\, g\in Q_8\}$ (see, e.g., [7]). As in [24], ensuing partitions of $\mathbb{H}^d$ are by ``quaternionic cubical wedges'' $\{\mathcal{W}\}_{g\in Q_8}$ centered about quaternionic hyperplanes. These can be expressed explicitly by $\mathcal{W}_g=\{\mathbf{u}\in \mathbb{H}^d\mid \, (\exists v\in V_g) \, \, \langle \mathbf{u},\mathbf{a}\rangle_{\mathbb{H}}-\bar{b} = v\}$ for each $(\mathbf{a},b) \in S(\mathbb{H}^{d+1})$, where $\langle \mathbf{u},\mathbf{a}\rangle_{\mathbb{H}}=\sum_{i=1}^n u_i\bar{a}_i$ is the standard quaternion-valued inner product. . 
		
	The representation theory of $Q_8$ is well-known: 1-dimensional representations given by the compositions $\chi_\epsilon: Q_8\rightarrow Q_8/\{\pm 1\}\cong Z_2^2\rightarrow U(1)$, $\epsilon\in\mathbb{Z}_2\oplus \mathbb{Z}_2$, and the 2-dimensional representation $\sigma: Q_8\hookrightarrow S(\mathbb{H}) \cong SU(2)$. Thus the Fourier expansion for any wedge decomposition and a given complex measure $\mu=\mu_1+i\mu_2$ is 
	\begin{equation} \mu(\mathcal{W}_g) = \sum_{\epsilon\in\mathbb{Z}_2^2}c_\epsilon \chi_\epsilon(g) + 2\, \Trace(c_\sigma \sigma(g)) \end{equation}
		
	As with finite abelian groups,  the cohomology of $Q_8$ (given in Section 7.1.2) precludes the equipartition of arbitrary $\mu$ by annihilating all transforms except $c_{(0,0)}$. Nonetheless,  all transforms but $c_{(0,0)}$ and $c_{(1,0)}$ can be made to vanish if $d\geq 3$, in which case $\mu(\mathcal{W}_g) = \frac{1}{8}\mu(\mathbb{H}^d)+ c_{(1,0)}$ for $g\in Z_4=\{\pm1,\pm i\}$ and $\mu(\mathcal{W}_g) = \frac{1}{8}\mu(\mathbb{H}^d)-c_{(1,0)}$ for $g\in jZ_4=\{\pm j, \pm k\}$:
	
\begin{proposition} Any two measures $\mu_1,\mu_2$ on $\mathbb{H}^3$ can be equipartitioned modulo $Z_4$ by a quaternionic cubical wedge partition $\{\mathcal{W}_g\}_{g\in Q_8}$. Thus \begin{equation} \mu_\ell(\mathcal{R}_{r,s})=\mu_\ell(\mathbb{H}^3)/4 \end{equation} for each $\ell=1,2$ and each union of wedges $\mathcal{R}_{r,s}=\mathcal{W}_{i^r}\cup \mathcal{W}_{i^sj}$, $0\leq r,s<4$.   \end{proposition}

	For comparison, any two measures on $\mathbb{C}^6\cong\mathbb{H}^3$ can also be equipartitioned mod $Z_4<Z_8$, i.e., by each regular 4-fan composing a complex regular 8-fan (by annihilating each $c_\epsilon$, $\epsilon\neq 0,4$, for a given complex measure). 
	
\section{Proofs of Theorem 3.1 and Proposition 6.1}

	We follow the \textit{configuration-space/test-map paradigm} [32], the established method for the topological reduction of problems in combinatorial and discrete geometry. For a $n$-tuple of complex measures $\mu=(\mu_1,\ldots,\mu_n)$, evaluating the  $(\rho,G)$-average of the given $\rho:G\rightarrow U(n)$ produces a continuous (test-) map $\mathcal{F}_\mu: X\rightarrow \mathbb{C}^n$, where the (configuration-) space $X$ is a free $G$-manifold which includes all the regions of all the non-trivial $G$-decompositions. Crucially, this map is $G$-equivariant, so that the balancing of these measures, represented by a zero of this map, is guaranteed by a Borsuk-Ulam type result (Proposition 7.1).  		
	
\begin{proof}  The discussion in Section 3 shows that $[S(\mathbb{C}^{d+1})]^k$ realizes all possible regions of all possible fan partitions $F_{q_1},\ldots, F_{q_k}$ of $\mathbb{C}^d$, including those at infinity. To ensure continuity, however, we remove from each coordinate sphere $S(\mathbb{C}^{d+1})$ the copy of $Z_{q_j}$ lying in $0 \times S^1$. Thus $X=\Pi_{j=1}^k X_j$, $X_j=S(\mathbb{C}^{d+1})- Z_{q_j}$. As before, $G$ acts freely on $X$, and each $G$-orbit of $x=(\textbf{a}_1,b_1,\ldots, \textbf{a}_k,b_k)\in X$ determines the sets $\mathcal{R}_g (x)= \{\mathbf{u}\in\mathbb{C}^d\mid (\forall \, 1\leq j \leq k)\,(\exists v_j \in R_{r_j}(q_j))\,\langle \mathbf{u},\mathbf{a}_j \rangle_{\mathbb{C}} + \bar{b}_j = v_j \}$, including all the regions of all non-trivial fan partitions. For $G=Q_8$ we let $X=S(\mathbb{H}^{d+1})-(0\times Y)$, where $Y=\cup_{g\in Q_8} \partial C'_g$,  $C'_g$ being the intersection of $S^3$ with the cone $V_g$ of Section 6. Again, the sets $\mathcal{R}_g(x)=\{\mathbf{u}\in \mathbb{H}^d\mid \, (\exists v\in V_g) \langle \mathbf{u},\mathbf{a}\rangle_{\mathbb{H}}+\bar{b} = v\}$ for $x=(\mathbf{a},b)$ include all the non-trivial cubical wedge decompositions, and $G$ acts freely on $X$ as before. 
	
		For Theorem 3.1, we seek to balance $\mu=(\mu_1,\ldots,\mu_n)$ on $\mathbb{C}^d$ by $\rho: G\rightarrow U(n)$, while for Proposition 6.1 we seek to balance the 6-tuple $\mu=(\mu_1,\mu_1,\mu_1, 0,0,\mu_1)$ by $\rho=\chi_{(1,1)}\oplus\chi_{(0,1)}\oplus 2\sigma: Q_8\rightarrow U(6)$, where $\mu_1$ is a single complex measure on $\mathbb{H}^3$.  In either case, we define the map $\mathcal{F}_\mu: X\rightarrow \mathbb{C}^n$ by \begin{equation} \mathcal{F}_\mu(x)=\frac{1}{|G|}\sum_{g\in G} \rho_g^{-1}(\mu(\mathcal{R}_g(x)) \end{equation} For finite abelian $G$, the exclusion of each $Z_{q_j}$ guarantees that each $F_{q_j}(x)=\{\mathbf{u}\in\mathbb{C}^d\mid (\exists\, 0\leq r < q_j)\, \langle \mathbf{u},\mathbf{a}_j \rangle_{\mathbb{C}} + \bar{b}_j = \zeta_{q_j}^r\}$ is either a complex regular $q_j$-fan (if $\mathbf{a}_j\neq0$) or the empty set (if $\mathbf{a}_j=0$), and therefore that each  $\partial\mathcal{R}_g (x): = \{\mathbf{u}\in\mathbb{C}^d\mid (\exists \, 1\leq j \leq k)\,(\exists v_j \in \partial R_{r_j}(q_j))\,\langle \mathbf{u},\mathbf{a}_j \rangle_{\mathbb{C}}+ \bar{b}_j = v_j \}$ has measure zero. A dominated convergence argument as in the proof of the Ham Sandwich Theorem in [18] or in [24] shows that $x\mapsto \mu(\mathcal{R}_g(x))$ is continuous. The $Q_8$-case  is proven similarly [24]. 
		
		As $\mathcal{R}_{g_1}(g_2x)=\mathcal{R}_{g_1g_2}(x)$ for all $x\in X$ and all $g_1,g_2\in G$, the map $\mathcal{F}_\mu$ is $G$-equivariant, i.e, $\mathcal{F}_\mu(g\cdot x) = \rho_g(\mathcal{F}_\mu (x))$ for all $x\in X$ and $g\in G$, so by Proposition 7.1 below there exists some $x\in X$ with $\mathcal{F}_\mu(x)=0$. Noting that it suffices to assume $\mu\neq 0$ (otherwise any non-trivial partition will do), we show lastly that the $G$-orbit of such an $x$ determines a non-trivial partition. For Theorem 3.1, one has $x\in (0\times S^1)^k$ otherwise, $b_j\notin Z_{q_j}$, so there exists some $g_0\in G$ for which $\mathcal{R}_g(x)= \emptyset$ if  $g\neq g_0$ and $\mathcal{R}_{g_0}(x)=\mathbb{C}^d$.  Hence $\mathcal{F}_\mu (x)=\rho_{g_0}^{-1}(\mu(\mathbb{C}^d))\neq 0$. Again, the $Q_8$-argument is identical [24]. \end{proof}

\begin{proposition} Let $X$ be the spaces in the proofs of Theorem 3.1 and Proposition 6.1 above, and let  $\rho: G\rightarrow U(n)$ be their respective representations, i.e., $\rho=\oplus_{r=1}^n \chi^{\epsilon_r}: \Pi_{j=1}^kZ_{q_j}\rightarrow U(n)$ from (3.3), and $\rho=\chi_{(1,1)}\oplus \chi_{(0,1)}\oplus 2\sigma: Q_8\rightarrow U(6)$. Then for any continuous $G$-equivariant map $h: X\rightarrow \mathbb{C}^n$, there exists some $x\in X$ such that $h(x)=0$. 
\end{proposition}

\begin{proof} We proceed in a standard fashion: quotienting $X\times \mathbb{C}^n$ by the diagonal $G$-action gives a  complex vector bundle $\mathbb{C}^n \hookrightarrow E:=X \times_G \mathbb{C}^n \rightarrow \overline{X}:=X/G$, and a zero of the section $s: \overline{X}\rightarrow E$ induced from $x\mapsto (x, h(x))$ is equivalent to a zero of $h$. As a non-vanishing section implies a zero top Chern class $c_n(E)\in H^{2n}(\overline{X};\mathbb{Z})$ (see, e.g., [19]), we show that $c_n(E)\neq 0$.

	 In each case, $E$ is the pullback of $\mathbb{C}^n\hookrightarrow E_\rho:=EG\times_G \mathbb{C}^d \rightarrow BG$ under the inclusion $i: \overline{X}\hookrightarrow BG$, where $EG$ and $BG$ are the total space and classifying space, respectively, of the universal bundle $G\hookrightarrow EG \rightarrow BG$ for the group $G$ (see, e.g., [12, 13]). We recall that $BG$ is unique up to homotopy. By naturality, the total Chern class $c(E)$ is $i^*(c(\rho))$, where $c(\rho):=c(E_\rho)$ is the ``total Chern class of the representation" [1]. In fact, $c_n(E)\neq0$ if $c_n(\rho)\neq0$ in the cases considered, so one is reduced to the calculation of  $c_n(\rho)\in H^*(BG;\mathbb{Z})$ given below.\end{proof} 

\subsection{Chern Class Calculations}

The explicit calculation of $c(\rho)$ for general $\rho: G\rightarrow U(n)$ can be very complicated (see, e.g., [8]), though here the computations are essentially classical, which we sketch nonetheless for the sake of completeness. Recall that for any paracompact space $B$, evaluating the first Chern class gives an isomorphism \begin{equation} c_1: Vect_{\mathbb{C}}^1(B) \stackrel{\cong}{\rightarrow} H^2(B;\mathbb{Z}),  \end{equation} where the space $Vect_{\mathbb{C}}^1(B)$ of all complex line bundles over $B$ is a group under tensor products. Thus the isomorphism (7.2) can be written as $c_1(E_1\otimes E_2)=c_1(E_1)+c_1(E_2)$ (see, e.g., [13]).

\subsubsection{$G=\Pi_{j=1}^kZ_{q_j}$} Recall that $BZ_q$ can be identified with the infinite-dimensional Lens space $L^\infty(q)=S(\mathbb{C}^\infty)/Z_q$, the union of the finite dimensional dimensional Lens spaces $L^{2d-1}(q)=S(\mathbb{C}^d)/Z_q$ given by the standard $Z_q$-action, and hence that $BG=BZ_{q_1}\times\ldots\times BZ_{q_k}$ can be seen as their product.

	For our space $\overline{X}=\Pi_{j=1}^k \overline{X}_j$, each $\overline{X}_j=L^{2d+1}(q_j)-pt$ deformation retracts onto the $2d$-skeleton of $BZ_{q_j}$. It will suffice to consider only the (far simpler, see, e.g., [6]) tensor-subrings of $H^*(BG)$ and $H^*(\overline{X})$, i.e., the images of $\otimes_{j=1}^n H^*(Z_{q_j})\rightarrow H^*(BG)$ and $\otimes_{j=1}^n H^*(\overline{X}_j)\rightarrow H^*(\overline{X})$  induced by projections (these are injections by the general K\"unneth formula [12]). Since $H^*(BZ_q)=\mathbb{Z}[b]/(qb)$, $b=c_1(\chi_1)$, $H^*_{tensor}(BG)=\mathbb{Z}[b_1, \ldots, b_k]/(q_1b_1,\ldots, q_kb_k)$, where $b_j=c_1(\chi_{e_j})$ and $e_j$ is  the $j$-th basis vector of $\oplus_{j=1}^n \mathbb{Z}_{q_j}$. By cellular cohomology, each restriction $i_j^*: H^*(BZ_{q_j}) \rightarrow H^*(\overline{X}_j)$ is an isomorphism in dimensions $d'\leq 2d$ and is the zero-map otherwise, so that $(b_1^{d+1}, \ldots, b_k^{d+1})$ is the  kernel of $i^*: H_{tensor}^*(BG)\rightarrow H_{tensor}^*(\overline{X})$.  On the other hand, $c(\rho)=c(\oplus_{r=1}^n \chi_{\epsilon_r})=\Pi_{r=1}^n(1+ c_1(\chi_{\epsilon_r}))$ by the Whitney sum formula [19]. As $\chi_{\epsilon_r} = \otimes_{j=1}^k \otimes^{\epsilon_{r,j}}\chi_{e_j}$, $c_1(\chi_{\epsilon_r})= \epsilon_{r,1}b_1+\ldots +\epsilon_{r,k}b_k$ by (7.2), and therefore $c_n(\rho)=\Pi_{r=1}^n (\epsilon_{r,1}b_1+\ldots +\epsilon_{r,k}b_k)=f(b_1,\ldots, b_k)$ is precisely the polynomial (3.3), which is not in $(b_1^{d+1},\ldots, b_k^{d+1})$ by assumption.  

\subsubsection{$G=Q_8$} $H^*(BQ_8;\mathbb{Z})= \mathbb{Z}[\alpha, \beta, \gamma]/\mathcal{I}$, where $\mathcal{I}=(2\alpha, 2\beta, 8\gamma, \alpha^2,\beta^2, \alpha\beta-4\gamma)$ and $|\alpha|=|\beta|=2$, $|\gamma|=4$ (see, e.g., [1]). One can identify $\alpha$ with $c_1(\chi_{(1,0)})$ and  $\beta$ with $c_1(\chi_{(0,1)})$, while that $c_2(\sigma)$ is a generator of $H^4(BQ_8;\mathbb{Z})\cong\mathbb{Z}_8$ follows by showing it to be non-zero mod 2, or in other words that the top Stiefel-Whitney class $w_4(\sigma)\in H^4(BQ_8;\mathbb{Z}_2)$ of the underlying real bundle of $E(\sigma)$ is non-zero. This was done in [21, formula 5.2]. As $\chi_{(1,1)}=\chi_{(1,0)}\otimes\chi_{(0,1)}$, it follows that $c_6(\rho)=(\alpha+\beta)\beta \gamma^2=4\gamma^3\neq0$. On the other hand, $X_0:=S^3-Y$ equivariantly deformation retracts onto $Q_8=\{\pm 1, \pm i, \pm j, \pm k\}\subset S^3$, so $\overline{X}=[S(\mathbb{H}^3) \ast X_0]/Q_8$ is homotopy equivalent to $[S(\mathbb{H}^3)/Q_8] \ast pt$, the 12-skeleton of $BQ_8$, and again $i^*: H^*(BQ_8)\rightarrow H^*(\overline{X})$ is an isomorphism in dimensions $d'\leq 12$. 

\section{Acknowledgments} 

This research was partially supported by ERC advanced grant 32094 during visits with Gil Kalai at the Hebrew University of Jerusalem, whom the author thanks for many helpful conversations.

\bibliographystyle{plain}

\begin{thebibliography}{10}

\bibitem{} M.F. Atiyah. Characters and Cohomology of Finite Groups, \textit{Inst. Hautes \v Etudes Sci. Publ. Math.}, Vol. 9 (1961) 23--64. 

\bibitem{ } I. B\'ar\'any and J. Matou\v sek. Simultaneous Partitions of Measures by $k$-Fans, \textit{Discrete Comput. Geom.}, Vol. 25 (2001) 317--334.

\bibitem{ } P. Blagojevi\'c and R. Karasev. Extensions of Theorems of Rattray and Makeev, \textit{Top. Methods in Nonlinear Anal.}, Vol. 40, No. 1 (2012) 189--213.

\bibitem{ } P. Blagojevi\'c and G. Ziegler. Convex Equipartitions via Equivariant Obstruction Theory, \textit{Israel J. Math}, Vol. 200 (2014) 49--77.  

\bibitem{ } P. Blagojevi\'c and G. Ziegler. The Ideal-Valued Index for a Dihedral Group Action, and Mass Partition by Two Hyperplanes, \textit{Topology Appl.} Vol. 158, No. 12 (2011) 1326-1351. 

\bibitem{ } G.R. Champan. The Cohomology Ring of a Finite Abelian Group, \textit{Proc. London Math. Soc.}, Vol. 3 (1982) 564--576. 

\bibitem{ } P. Du Val. \textit{Homographies, Quaternions, and Rotations}, Oxford Mathematical Monographs,
1964.

\bibitem{ } L. Evens. On the Chern Classes of Representations of Finite Groups, \textit{Trans. Amer. Math. Soc.}, Vol. 115 (1965), 180--193.

\bibitem{ } N.J. Fine. Binomial Coefficients Modulo a Prime, \textit{Amer. Math. Monthly}, Vol. 54 (1947)  589--592.

\bibitem{ } B. Gr\"unbaum. Partitions of Mass-Distributions and Convex Bodies by Hyperplanes, \textit{Pacific J. Math.}, Vol. 10 (1960) 1257--1261.

\bibitem{ } H. Hadwiger. Simultane Vierteilung Zweier K\"orper. \textit{Arch. Math.} (Basel) Vol. 17 (1966) 274--278.

\bibitem{ } A. Hatcher. \textit{Algebraic Topology}, Cambridge University Press (2002).

\bibitem{ } D. Husemoller. \textit{Fiber Bundles}, Springer-Verlag (1994). 

\bibitem{ }  R. Karasev, A. Hubard, and B. Aronov. Convex Equipartitions: The Spicy Chicken Theorem, \textit{Geom. Dedicata}, Vol. 170 (2014), 263--279.

\bibitem{ }  R. Karasev. Equipartition of a Measure by $Z_p^k$-Invariant Fans, \textit{Discrete Comput. Geom.}, Vol. 43 (2010) 477--481. 

\bibitem{ } V.V. Makeev. Equipartition of a Continuous Mass Distribution, \textit{J. Math. Sci.}, Vol 140, No. 4 (2007) 551--557.

\bibitem{ } P. Mani-Levitska, S. Vre\'cica, and R. \v Zivaljevi\'c. Topology and Combinatorics of Partitions of Masses by Hyperplanes, \textit{Adv. Math.}, Vol. 207 (2006) 266--296.

\bibitem{ } J. Matou\v sek. \textit{Using the Borsuk-Ulam Theorem. Lectures on Topological Methods in Combinatorics and Geometry}, Springer-Verlag (2008).

\bibitem{ } J.W. Milnor and J.D. Stasheff. \textit{Characteristic Classes}, Princeton University Press (1974).

\bibitem{ }  D. Quillen. The Mod 2 Cohomology Rings of Extra-Special 2-Groups and the Spinor Groups, \textit{Math. Ann.}, Vol. 194, No. 3 (1971) 197--212. 

\bibitem{} E. A. Ramos. Equipartition of Mass Distributions by Hyperplanes, \textit{Discrete Comput. Geom.}, Vol. 16 (1996) 147--167.

\bibitem{ } J.P. Serre. \textit{Linear Representations of Finite Groups.} Springer-Verlag (1977). 

\bibitem{} S. Simon. Average-Value Tverberg Partitions via Finite Fourier Analysis. arXiv:1501.04612 [math.CO]

\bibitem{} S. Simon. $G$-Ham Sandwich Theorems: Balancing Measures by Finite Subgroups of Spheres, \textit{J. Comb. Theory Ser. A}, Vol. 120 (2013) 1906--1912.

\bibitem{ } S. Simon. Measure Partitions via Fourier Analysis II: Center Transversality in the $L^2$-norm for Complex Hyperplanes, \textit{In preparation}.

\bibitem{} A. Terras. \textit{Fourier Analysis on Finite Groups and Applications}, London Mathematical Society (1999). 

\bibitem{ } S.T. Vre\'cica and R.T. \v Zivaljevi\'c. Conical Equipartitions of Mass Distributions, \textit{Discrete Comput. Geom.}, Vol. 25 (2001) 335--350.

\bibitem{ } S.T. Vre\'cica and R.T. \v Zivaljevi\'c. The Ham-Sandwich Theorem Revisited, \textit{Israel J. Math.}, Vol. 78 (1992) 21--32.

\bibitem{ } C. Wilkerson. A Primer on the Dickson Invariants, \textit{Proceedings of the Northwestern Homotopy Theory Conference, Contemp. Math. 19} (1983) 421--434. 

\bibitem{} F. Yao, D. Dobkin, H. Edelsbrunner, M. Paterson. Partitioning Space for Range Queries, \textit{SIAM J. Comput.}, Vol. 18, No. 2 (1989) 371--384.

\bibitem{ } R.T. \v Zivaljevi\'c. Computational Topology of Equipartitions by Hyperplanes. arXiv:1111.1608v3 [math.MG] 

\bibitem{ } R. \v Zivaljevi\'c. Topological Methods. Chapter 14 in \textit{Handbook of Discrete and Computational Geometry}, J.E. Goodman, J. O'Rourke, eds, Chapman \& Hall/CRC (2004) 305--330.

\bibitem{ } R. \v Zivaljevi\'c. User's Guide to Equivariant Methods in Combinatorics, \textit{Publ. Inst. Math. Belgrade}, Vol.59, No. 73 (1996) 114--130. 

\bibitem{ } R. \v Zivaljevi\'c. User's Guide to Equivariant Methods in Combinatorics II, \textit{Publ. Inst. Math. Belgrade}, Vol.64, No. 78 (1998) 107--132.

\end{thebibliography}

\end{document}